\documentclass{amsart}
\usepackage{amsmath,amsfonts,amssymb}
\usepackage{graphicx}
\usepackage{amsthm}
\usepackage{ifthen}
\usepackage{longtable}
\usepackage{array}
\usepackage{url}
\usepackage{multirow}

\newcommand{\Z}{\mathbb{Z}}
\newcommand{\Q}{\mathbb{Q}}

\newcommand{\vect}[1]{\mathbf{#1}}

\newtheorem{definition}{Definition}
\newtheorem{theorem}{Theorem}
\newtheorem{lemma}{Lemma}
\newtheorem{corollary}{Corollary}

\begin{document}
\title[Uniform distribution of multidimensional of LS-sequences]{On the uniform distribution modulo 1 of multidimensional LS-sequences}

\author[C. Aistleitner]{Christoph Aistleitner}
\address{C. Aistleitner \newline
\indent University of New South Wales, \newline
\indent Department of Applied Mathematics,\newline
\indent  School of Mathematics and Statistics,\newline
\indent  Sydney NSW 2052, Australia.}
\email{aistleitner\char'100math.tugraz.at}

\author[M. Hofer]{Markus Hofer}
\address{M. Hofer \newline
\indent Graz University of Technology, \newline
\indent Institute of Mathematics A,\newline
\indent  Steyrergasse 30,\newline
\indent  8010 Graz, Austria.}
\email{markus.hofer\char'100tugraz.at}

\author[V. Ziegler]{Volker Ziegler}
\address{V. Ziegler \newline
\indent Institute of Mathematics A,\newline
\indent Graz University of Technology \newline
\indent Steyrergasse 30, \newline
\indent A-8010 Graz, Austria.}
\email{ziegler\char'100math.tugraz.at}

\thanks{The first author is supported by Project P24302 and an Erwin Schr\"odinger Fellowship of the Austrian Science Fund (FWF)}

\begin{abstract}
Ingrid Carbone introduced the notion of so-called LS-sequences of points, which are obtained by a generalization of Kakutani's interval splitting procedure. Under an appropriate choice of the parameters $L$ and $S$, such sequences have low discrepancy, which means that they are natural candidates for Quasi-Monte Carlo integration. It is tempting to assume that LS-sequences can be combined coordinatewise to obtain a multidimensional low-discrepancy sequence. However, in the present paper we prove that this is not always the case: if the parameters $L_1,S_1$ and $L_2,S_2$ of two one-dimensional low-discrepancy LS-sequences satisfy certain number-theoretic conditions, then their two-dimensional combination is not even dense in $[0,1]^2$.
\end{abstract}

\maketitle

\section{Introduction and statement of results}

For two points $a,b \in [0,1)^d$  we write $a \leq b$ and $a<b$ if the corresponding inequalities hold in each coordinate; furthermore, we write $[a,b)$ for the set $\{x \in [0,1)^d: ~a \leq x < b\}$, and call such a set a ($d$-dimensional) interval. We denote by $\mathbf{1}_{I}$ the indicator function of the set $I \subseteq [0,1)^d$ and by $\lambda_d$ the $d$-dimensional Lebesgue measure. We will sometimes write $0$ for the $d$-dimensional vector $(0,\dots,0)$.\\

A sequence $(x_n)_{n \in \mathbb{N}}$ of points in $[0,1]^d$ is called \textit{uniformly distributed modulo 1} (u.d. mod 1) if
\begin{equation*}
 \lim_{N \rightarrow \infty} \frac{\sum_{n = 1}^N \mathbf{1}_{[a, b)} (x_n)}{N} = \lambda_d([a,b))
\end{equation*}
for all $d$-dimensional intervals $[a, b ) \subseteq [0,1)^d$. A further characterization of u.d.\ is given by the following well-known result of Weyl~\cite{weyl}: a sequence $(x_n)_{n \in \mathbb{N}}$ of points in $[0,1)^d$ is u.d. mod 1 if and only if for every continuous function $f$ on $[0,1)^d$ the relation
\begin{equation*}
  \lim_{N \rightarrow \infty} \frac{\sum_{n = 1}^N f(x_n)}{N} = \int_{[0,1)^d} f(x) dx
\end{equation*}
holds. Although this theorem shows the possibility of using u.d.\ point sequences for numerical integration --- a method usually called \emph{Quasi-Monte Carlo (QMC) integration} --- it does not give any information on the integration error.\\

The Koksma--Hlawka inequality~\cite{hlawka} states that this integration error can be bounded by the product of the variation of $f$ (in the sense of Hardy and Krause), denoted by $V(f)$, and the so-called \emph{star-discrepancy} $D_N^*$ of the point sequence $(x_n)_{n \in \mathbb{N}}$. More precisely,
\begin{equation*}
 \left| \frac{1}{N} \sum_{n = 1}^N f(x_n) - \int_{[0,1]^d} f(x) dx \right| \leq V(f) D_N^*(x_n),
\end{equation*}
where $D_N^*$ is given by
\begin{equation*}
 D_N^* = D_N^*(x_1, \ldots, x_N) = \sup_{a \in [0,1)^d} \left| \frac{\sum_{n=1}^N \mathbf{1}_{(0,a)} (x_n)}{N} - \lambda_d([0,a))\right|.
\end{equation*}
The Koksma--Hlawka inequality suggest that for QMC integration one should use a sequence of points whose discrepancy is as small as possible. The best known point sequences achieve a discrepancy of order $\mathcal{O}(N^{-1} (\log N)^d)$; such sequences are called \emph{low-discrepancy sequences}. Note that this convergence rate is for all $d \geq 1$ better than that of the probabilistic error of Monte Carlo integration, where a sequence of random points is used instead of a deterministic one. QMC integration can be successfully applied in several different areas of applied mathematics, for example in actuarial or financial mathematics, where frequently high-dimensional integration problems arise (see e.g.~\cite{albrecher,okt}). For more information on discrepancy theory, low-discrepancy sequences and QMC integration see~\cite{dt,kn}.\\

A popular approach to construct $d$-dimensional low-discrepancy sequences is to combine $d$ one-dimensional low-discrepancy sequences; this works for example for the so-called Halton sequence, which is obtained by joining one-dimensional van der Corput sequences coordinatewise. In the present paper we show that this construction principle is not generally applicable for a special class of one-dimensional low-discrepancy sequence, so-called LS-sequences. We prove that the limit distribution of a multidimensional LS-sequences (composed coordinatewise from one-dimensional low-discrepancy LS-sequences) can spectacularly fail to be u.d., if there is a certain number-theoretic connection between the parameters of the one-dimensional sequences. To explain the construction of LS-sequences we need some definitions.

\begin{definition}[Kakutani splitting procedure]
If $\alpha \in (0,1)$ and $\pi = \{[t_{i-1}, t_i): 1 \leq i \leq k\}$ is any partition of $[0,1)$, then $\alpha \pi$ denotes its so-called $\alpha$-refinement, which is obtained by subdividing all intervals of $\pi$ having maximal length into two parts, proportional to $\alpha$ and $1- \alpha$, respectively. The so-called Kakutani's sequence of partitions $(\alpha^n \omega)_{n \in \mathbb{N}}$ is obtained as the successive $\alpha$-refinement of the trivial partition $\omega = \{[0,1)\}$.
\end{definition}

The notion of $\alpha$-refinements can be generalized in a natural way to so-called $\rho$-refinements.

\begin{definition}[$\rho$-refinement]
Let $\rho$ denote a non-trivial finite partition of $[0,1)$. Then the $\rho$-refinement of a partition $\pi$ of $[0,1)$, denoted by $\rho \pi$, is given by subdividing all intervals of maximal length positively homothetically to $\rho$. Note that the $\alpha$-refinement is a special case with $\rho = \{[0, \alpha), [\alpha , 1)\}$.
\end{definition}

By a classical result of Kakutani~\cite{kakutani}, for any $\alpha$ the sequence of partitions  $(\alpha^n \omega)_{n \in \mathbb{N}}$ is uniformly distributed, which means that for every interval $[a,b] \subset [0,1]$,
\begin{equation*}
 \lim_{n \rightarrow \infty} \frac{\sum_{i=1}^{k(n)} \mathbf{1}_{[a,b]}(t_i^n)}{k(n)} = b-a,
\end{equation*}
where $k(n)$ denotes the number of intervals in $\alpha^n \omega=\{[t_{i-1}^n,t_i^n),~1 \leq i \leq k(n)\}$. The same result holds for any sequence of $\rho$-refinements of $\omega$, due to a result of Vol\v{c}i\v{c}~\cite{volcic} (see also~\cite{ah,drmota}). A multidimensional generalization of $\rho$-refinements has been introduced by Carbone and Vol\v{c}i\v{c}~\cite{carbone2}. A special case of a $\rho$-refinement is the so-called \emph{LS-sequence of partitions}. This sequence of partitions has been introduced by Carbone~\cite{carbone}.

\begin{definition}[LS-sequence of partitions]
An LS-sequence of partitions  \linebreak $(\rho^n_{L,S} \omega)_{n \in \mathbb{N}}$ is the successive $\rho$-refinement of the trivial partition $\omega$, where $\rho_{L,S}$ consists of $L+S$ intervals such that the first $L>0$ intervals of $\rho_{L,S}$ have length $\beta$ and the remaining $S \geq 0$ intervals have length $\beta^2$. Note that necessarily $L\beta+S\beta^2 = 1$ holds, and consequently for each pair $(L,S)$ of parameters there exists exactly one corresponding number $\beta$.
\end{definition}

It can easily be seen that for every $n$ the partition $\rho^n_{L,S} \omega$ consists only of intervals having either length $\beta^n$ or $\beta^{n+1}$. This fact makes the analysis of LS-sequences relatively simple, in comparison to the analysis of general $\rho$-refinements. We denote by $t_n$ the total number of intervals of $\rho^n_{L,S} \omega$, and correspondingly let $l_n$ and $s_n$ be the number of long and short intervals after $n$ steps, respectively (more precisely, $l_n$ is the number of intervals of length $\beta^n$, and $s_n$ is the number of intervals of length $\beta^{n+1}$). It is easy to see that these sequences satisfy the following recurrence relations (see~\cite{carbone}):
\begin{align*}
 t_n &= L t_{n - 1} + S t_{n - 2},\\
 l_n &= L l_{n - 1} + S l_{n - 2},\\
 s_n &= L s_{n - 1} + S s_{n - 2},
\end{align*}
for $n \geq 2$, where $t_1 = L + S, t_0 = 1, l_1 = L, l_0, = 1, s_1 = S$ and $s_0 = 0$. Solving these binary recurrences we obtain explicit formulas for $t_n,l_n$ and $s_n$:
\begin{align}\label{Rec:tk}
t_n&=\tau_0 \beta^{-n}+\tau_1(-S\beta)^n,& \tau_0&=\frac{L+2S+\sqrt{L^2+4S}}{2\sqrt{L^2+4S}},\\\nonumber
& & \tau_1&=\frac{-L-2S+\sqrt{L^2+4S}}{2\sqrt{L^2+4S}},\\\label{Rec:lk}
l_n&=\lambda_0 \beta^{-n}+\lambda_1(-S\beta)^n,& \lambda_0&=\frac{L+\sqrt{L^2+4S}}{2\sqrt{L^2+4S}},\\\nonumber
& & \lambda_1&=\frac{-L+\sqrt{L^2+4S}}{2\sqrt{L^2+4S}},\\\label{Rec:sk}
s_n&=\sigma_0 \beta^{-n}+\sigma_1(-S\beta)^n,& \sigma_0&=\frac{2S+\sqrt{L^2+4S}}{2\sqrt{L^2+4S}},\\\nonumber
& & \sigma_1&=\frac{-2S+\sqrt{L^2+4S}}{2\sqrt{L^2+4S}}.
\end{align}

We can generate a sequence of points from a sequence of partitions by ordering the left endpoints of the intervals in the partition. The following rule by Carbone~\cite{carbone} defines the so-called LS-sequence of points.

\begin{definition}[LS-sequence of points]\label{Def:LSpoints}
Given an LS-sequence of partitions \linebreak $(\rho^n_{L,S} \omega)_{n \in \mathbb{N}}$ we define the corresponding LS-sequence of points $(\xi_{L,S}^n)_{n \in \mathbb{N}}$ as follows: the first $t_1$ points are the left endpoints of the partition $\rho_{L,S} \omega$ ordered by magnitude. We denote this ordered set of points by $\Lambda^1_{L,S}$.

For $n > 1$ we define $\Lambda^{n + 1}_{L,S} = \{ \xi_{L,S}^0, \ldots, \xi_{L,S}^{t_{n + 1}-1} \}$ inductively as the ordered set of the left endpoints of the intervals of $\rho^n_{L,S} \omega$ in the following way:
\begin{align*}
\Lambda^{n + 1}_{L,S} = &\{ \xi_{L,S}^0, \ldots, \xi_{L,S}^{t_{n}-1},\\ 
&\psi_{1,0}^{n + 1}\left(\xi_{L,S}^0 \right), \ldots, \psi_{1,0}^{n + 1}\left(\xi_{L,S}^{l_n-1} \right), \ldots, \psi_{L,0}^{n + 1}\left(\xi_{L,S}^0 \right), \ldots, \psi_{L,0}^{n + 1}\left(\xi_{L,S}^{l_n-1} \right),\\
&\psi_{L,1}^{n + 1}\left(\xi_{L,S}^0 \right), \ldots, \psi_{L,1}^{n + 1}\left(\xi_{L,S}^{l_n-1} \right), \ldots, \psi_{L,S}^{n + 1}\left(\xi_{L,S}^0 \right), \ldots, \psi_{L,S}^{n + 1}\left(\xi_{L,S}^{l_n-1} \right) \},
\end{align*}
where
\begin{equation*}
 \psi^n_{i,j}(x) = x + i \beta^n + j \beta^{n + 1}.
\end{equation*}
\end{definition}
For more details on the definition of LS-sequences of points, and on the properties of such sequences, see~\cite{Carbonearxiv,carbone}.

Next, we recall the definition of the well-known van der Corput sequence in base $b \geq 2$, $b \in \mathbb{N}$. For every $n \in \mathbb{N}_0 $, the unique digit expansion of $n$ in base $b$ is given by
\begin{equation*}
 n = \sum_{i \geq 0} n_i b^i,
\end{equation*}
where $n_i \in \{0,1, \ldots, b-1\}, i \geq 0$.
For $n \in \mathbb{N}_0$ we define the \emph{radical-inverse function} (or \emph{Monna map}) $\phi_b(n) \colon \mathbb{N}_0 \rightarrow [0,1)$ by
\begin{equation}\label{monna}
\phi_b(n) = \phi_b \left( \sum_{i \geq 0} n_i b^i \right) := \sum_{i \geq 0} n_i b^{-i-1}.
\end{equation}
We call $x$ a $b$-adic rational if $x=a b^{-c}$, where $a$ and $c$ are positive integers and $0 \leq a < b^c$.
Note that $\phi_b(n)$ maps $\mathbb{N}$ onto the $b$-adic rationals in $[0,1)$, and therefore the image of $\mathbb{N}$ under $\phi_b(n)$ is dense in $[0,1)$.

\begin{definition}
The van der Corput sequence in base $b$ is defined as $(\phi_b(n))_{n \in \mathbb{N}}$.
\end{definition}

Note that the definition of the van der Corput sequence in base $b \geq 2$ coincides with the definition of the LS-sequence of points with parameters $L=b$ and $S=0$. Thus LS-sequences can be seen as a generalization of the van der Corput sequence. A remarkable property of van der Corput sequences is, that several van der Corput sequences in pairwise coprime bases can be combined coordinatewise to a multidimensional sequence, the so-called Halton sequence, which is a low-discrepancy sequence. As mentioned above this means that the discrepancy of a Halton sequence is of asymptotic order $\mathcal{O}(N^{-1} (\log N)^d)$, where $N$ is the number of points and $d$ denotes the dimension, which together with the Koksma-Hlawka inequality makes it a perfect candidate for Quasi-Monte Carlo integration (for details on the properties of van der Corput and Halton sequences, see~\cite{dt,kn}).\\

If we assume $S \geq 1$, then by a result of Carbone~\cite{carbone}, one-dimensional LS-sequences are low-discrepancy sequences if and only if $L > S - 1$. Thus it is tempting to assume that several LS-sequences can be combined coordinatewise in order to obtain multidimensional low-discrepancy sequences. If this was be the case, then this method would produce a new parametric class of multidimensional low-discrepancy sequences. However, even in the case of the combination of van der Corput sequences (which are a special case of LS-sequences, as mentioned before), the bases $b_1, \dots, b_d$ cannot be chosen arbitrarily, but have to satisfy a certain number-theoretic condition (they have to be coprime). A similar restriction can be expected in the case of combining LS-sequences.\\

In a talk in Graz in June 2012, Maria Rita Iac\`{o} presented several numerical examples of the asymptotic distribution of two-dimensional LS-sequences. In some cases they showed ``random'' behavior, while in others (for example when combining the sequence with parameters (1,1) and the sequence with parameters (4,1)) the distribution seemed to be rather erratic. Obviously the reason for this behavior is that there is a multiplicative relation between the solutions of the equations $x+x^2=1$ and $4x+x^2=1$, which define the lengths of the intervals for the LS-sequences with parameters (1,1) and (4,1), respectively. The purpose of the present paper is to prove that in fact the two-dimensional LS-sequences is \emph{not} uniformly distributed (and not even dense) in $[0,1]^2$ if such a multiplicative relation exists. Furthermore, in a second theorem we show that if the parameters of two one-dimensional LS-sequences have a greatest common divisor ($\gcd$) which is greater than 1, then the resulting two-dimensional LS-sequence is also not dense in $[0,1]^2$. This second result generalizes the requirement of having coprime bases of the van der Corput sequences, in order to obtain a low-discrepancy Halton sequence by joining them coordinatewise.\\

The formal definition of a multidimensional LS-sequence can be given as follows.

\begin{definition} \label{def6}
Let $\mathcal B=((L_1,S_1),\ldots,(L_d,S_d))$ be an ordered $d$-tuple of pairs, $(L_i,S_i)$ such that $L_i>0, ~S_i\geq 0$ and $L_i + S_i \geq 2$ for all $i$. Then we define the $d$-dimensional LS-sequence in base $\mathcal B$ as the sequence
\begin{equation*}
\vect{\xi}_\mathcal{B}^n=\left(\xi_{L_1,S_1}^n,\ldots,\xi_{L_d,S_d}^n\right)_{n \in \mathbb{N}}.
\end{equation*}
\end{definition}

The following theorem states that a two-dimensional LS-sequences in bases $\mathcal B=((L_1,S_1),(L_2,S_2))$, where the one-dimensional components are low-discrepancy sequences, is \emph{not} dense in $[0,1]^2$ if there exist integers $m$ and $k$ such that $\frac{\beta_1^{k+1}}{\beta_2^{m+1}}\in\Q$. For example, in the case $(L_1,S_1)=(1,1)$ and $(L_2,S_2)=(4,1)$ we have $\beta_2=\beta_1^3$.

\begin{theorem} \label{th1}
Let $\mathcal B=((L_1,S_1),(L_2,S_2))$ with $L_i> S_i-1\geq 0$ and assume that there exist integers $m$ and $k$ such that $\frac{\beta_1^{k+1}}{\beta_2^{m+1}}\in\Q$. Then the two-dimensional LS-sequence $\vect{\xi}_\mathcal{B}^n$ is not uniformly distributed, and not even dense in $[0,1]^2$.
\end{theorem}

On the other hand, we have not been able to derive any positive results, proving uniform distribution of a LS-sequence for an appropriate choice of $L_1,S_1$ and $L_2,S_2$ (except for the case of the Halton sequence). So up to date not a single example of parameters $L_1,S_1,L_2,S_2$ is known, for which either $S_1 \neq 0$ or $S_2 \neq 0$ and the corresponding two-dimensional LS-sequence is uniformly distributed.\\

Note that Theorem~\ref{th1} can also be applied to the multidimensional case, since for any multidimensional sequence of points, which is uniformly distributed, all lower-dimensional projections also have to be uniformly distributed. More precisely, we immediately get the following corollary.

\begin{corollary} \label{co1}
Let $\mathcal B=((L_1,S_1),\dots,(L_d,S_d))$ with $L_i> S_i-1\geq 0$ and assume that there exist numbers $u,w \in \{1, \dots, d\}$ and integers $m$ and $k$ such that $\frac{\beta_u^{k+1}}{\beta_w^{m+1}}\in\Q$. Then the $d$-dimensional LS-sequence $\vect{\xi}_\mathcal{B}^n$ is not uniform distributed, and not even dense in $[0,1]^d$.
\end{corollary}

The next theorem characterizes another class of two-dimensional LS-sequences, which are not dense in $[0,1]^2$.

\begin{theorem}\label{th2}
Let $\mathcal B=((L_1,S_1),(L_2,S_2))$ and assume that $\gcd(L_1, S_1, L_2, S_2) > 1$.  Then the two-dimensional LS-sequence $\vect{\xi}_\mathcal{B}^n$ is not dense in $[0,1]^2$.
\end{theorem}

Note that Theorem~\ref{th2} also includes the case of the Halton sequence. As above we can state a corollary which describes the $d$-dimensional situation.

\begin{corollary} \label{col2}
Let $\mathcal B=((L_1,S_1),\dots,(L_d,S_d))$ and assume that $\gcd(L_i, S_i, L_j, S_j) > 1$ for some $i,j \in \{1, \ldots, d\}, ~ i \neq j$.  Then the $d$-dimensional LS-sequence $\vect{\xi}_\mathcal{B}^n$ is not dense in $[0,1]^d$.
\end{corollary}

In the next section we provide several auxiliary results concerning one-dimensional LS-sequences. These lemmas will be essential in the proofs of Theorem~\ref{th1} and Theorem~\ref{th2}, which are presented in Section~\ref{Sec:ProofTh1}. 

\section{Points in elementary intervals}

Before we define elementary intervals and prove some of their properties, we prove the following recurrence relations for the sequences $t_n$ and $l_n$. 

\begin{lemma}\label{Lem:MixedRec}
We have
\[t_{n+1}=t_n+(L+S-1)l_n \quad \text{and}\quad l_{n+1}=t_n+(L-1)l_n\]
for all $n\geq 0$.
\end{lemma}

\begin{proof}
We prove this claim by induction on $n$. For $n=0$ we have
\[t_1=L+S=1+(L+S-1)=t_0+(L+S-1)l_0\]
and
\[l_1=L=1+(L-1)=t_0+(L-1)l_0.\]
Assume the claim is true for all indices $\leq n$. Then we have
\begin{multline*}
t_{n+1}=Lt_n+St_{n-1}=L(t_{n-1}+(L+S-1)l_{n-1})+S(t_{n-2}+(L+S-1)l_{n-2})=\\
Lt_{n-1}+St_{n-2}+(L+S-1)(Ll_{n-1}+Sl_{n-2})=t_n+(L+S-1)l_n
\end{multline*}
and
\begin{multline*}
l_{n+1}=Ll_n+Sl_{n-1}=L(t_{n-1}+(L-1)l_{n-1})+S(t_{n-2}+(L-1)l_{n-2})=\\
Lt_{n-1}+St_{n-2}+(L-1)(Ll_{n-1}+Sl_{n-2})=t_n+(L-1)l_n
\end{multline*}
which proves the lemma.
\end{proof}

We will also need a lemma on the irrationality of $\beta$.

\begin{lemma}\label{lem:irrational}
Let $L>S-1\geq 0$, then $\beta^k$ is irrational for every positive integer $k$. 
\end{lemma}

\begin{proof}
First, we prove that $\beta$ is irrational. In particular we have to prove that $L^2+4S$ is not a square. Therefore we note that
\[L^2<L^2+4S<L^2+4L+4=(L+2)^2,\]
thus $L^2+4S=(L+1)^2$, i.e. $S=\frac{2L+1}4\not\in \Z$. Hence $\beta$ is irrational.\\

Now suppose that $\beta^k$ is a rational for some $k>1$. Let us assume that $k$ is minimal. Then $\beta$ is an algebraic integer of degree $k$, which is a contradiction unless $k=2$, since $\beta$ is the solution of a quadratic equation. But $k=2$ yields
\[\beta^2=\frac{L^2+2S-L\sqrt{L^2+4S}}{2S^2}\in \Q,\]
which is not possible unless $L=0$.
\end{proof}

We call an interval \emph{elementary}, if it is an element of $\rho^n_{L,S} \omega$ for some $n$. Equivalently we can define elementary intervals as all intervals of the form $I_x^{(k)}=[\xi_{L,S}^x,\xi_{L,S}^x+\beta^k)$ for some $k$, where $x <l_k$. If $[\xi_{L,S}^x,\xi_{L,S}^x+\beta^k)$ is an elementary interval, then there necessarily exists an integer $y<t_k$ such that $\xi_{L,S}^x+\beta^k=\xi_{L,S}^y$. Note that in the case of the van der Corput sequence the elementary intervals defined in this way coincide with the usual ones.\\

In order to count points in elementary intervals we need a method to decide whether a point $\xi_{L,S}^N$ is contained in some given elementary interval or not. In the case of the van der Corput sequence this can be achieved by considering digit expansions as in \eqref{monna}, which motivates the following construction. Let $N\geq 0$ be a fixed integer and let $n$ be such that $t_n\leq N < t_{n+1}$. We construct two sequences $(\epsilon_k)_{0 \leq k \leq n}$ and $(\eta_k)_{0 \leq k \leq n}$ recursively in the following way:  We put 
\[N_n=N, \quad \epsilon_n=1, \quad\eta_n=\left\lfloor \frac{N_n-t_n}{l_n}\right\rfloor \quad \textrm{and} \quad N_{n-1}=N_n-t_n-\eta_nl_n.\]
For $k \leq n-1$, if $N_{k}<t_{k}$ we put $\epsilon_{k}=\eta_{k}=0$ and $N_{k-1}=N_{k}$. Otherwise we proceed as in the initial construction, i.e.\ put $\epsilon_k=1$, $\eta_k=\lfloor \frac{N_k-t_k}{l_k}\rfloor$ and $N_{k-1}=N_k-t_k-\eta_kl_k$. If $N_{k-1}=0$ we terminate and put $\epsilon_i=\eta_i=0$ for all $i<k$. Since $t_{k+1}=t_k+(L+S-1)l_k$ and $l_{k+1}\geq t_k\geq l_k$ for all $k\geq 0$, this algorithm yields a representation of $N$ in the form
\begin{equation}\label{Rep:N}
 N=\sum_{i=0}^n (\epsilon_it_i+\eta_il_i),
\end{equation}
where $\epsilon_i\in\{0,1\}$, $0\leq \eta\leq L+S-2$, and $\epsilon_i=0$ implies $\eta_i=0$. Furthermore, since $t_k+(L-1)l_k=l_{k+1}$, it follows that $\eta_i \geq L - 1$ implies $\epsilon_{i + 1} = 0$.\\ 
Note that the representation \eqref{Rep:N} is not unique. Consider e.g.\ the case $L=2,S=1$; then we have $t_2+l_2=12=t_2+t_1+l_1$. However, for the rest of this paper speaking of a representation we will always mean the representation whose coefficients $\epsilon_i$ and $\eta_i$ were constructed as explained in the algorithm above (i.e.\ in the above example the representation we choose is $12=t_2+l_2$).\\

In order to establish unique ``digit expansions`` we use the following Lemma.

\begin{lemma}\label{Lem:Digits}
There is a bijection between positive integers and finite sequences of the form
\[\mathcal D=((\epsilon_n,\eta_n),\ldots,(\epsilon_0,\eta_0))\]
such that $\epsilon_i\in\{0,1\}$, $\epsilon_n=1$, $0\leq \eta\leq L+S-2$, $\epsilon_i=0$ implies $\eta_i=0$ and $\eta_i\geq L-1$ implies $\epsilon_{i+1}=0$.
This bijection is given by
\[\Psi(\mathcal D)=\sum_{i=0}^n (\epsilon_it_i+\eta_il_i)\]
and its inverse
\[\Psi^{-1}(N)=((\epsilon_n,\eta_n),\ldots,(\epsilon_0,\eta_0)),\]
where the $\epsilon_i$ and $\eta_i$ are computed by the algorithm described above.
\end{lemma}

\begin{proof}
Let $N>0$ be an integer. Note that $\Psi^{-1}$ is injective since by construction we have $N_k<t_{k+1}$ and therefore in every step of the algorithm the pair $(\epsilon_k,\eta_k)$ is uniquely determined.\\

It remains to prove that $\Psi^{-1}(\Psi(\mathcal D))=\mathcal D$. We prove this by induction on the length $n+1$ of $\mathcal D$. The case $n=0$ is trivial, since $\Psi(\mathcal D)=\epsilon_0+\eta_0<t_1$, and applying $\Psi^{-1}$ yields indeed $\mathcal D$. Let us assume that the algorithm yields the correct sequence $\Psi(\mathcal D)$ for all sequences $\mathcal D$ of length $\leq n$. In particular, this implies $\Psi(\mathcal D)<t_n$ for all $\mathcal D$ of length $\leq n$. Assume now that $\mathcal D$ is of length $n+1$.\\

First, we prove that $N=\Psi(\mathcal D)<t_{n+1}$. Since $\epsilon_n=1$ we have $\eta_{n-1}\leq L-2$ and by induction
we know that for a sequence $\mathcal E$ of length $n$ starting with $(1,L-2)$ we have
$\Psi(\mathcal E)<t_{n-1}+(L-1)l_{n-1}=l_n$. Hence $\Psi(\mathcal D)<t_n+(L+S-2)l_n+l_n=t_{n+1}$.\\ 

Now let $\mathcal E$ be the sequence of length $n$ induced by $\mathcal D$ by deleting the entry $(\epsilon_n,\eta_n)$. Note that $\epsilon_{n-1}$ might be zero and thus $\mathcal E$ is not a valid output of the above algorithm. If $\epsilon_i = 0$ for all $i \leq n-1$, then the proof is trivial. Assume now that at least one $\epsilon_i > 0$ for $i \leq n-1$. Then we can write $\mathcal E = (\mathcal E_1,\mathcal E_2)$, where $\mathcal E_1 = ((\epsilon_{n-1},\eta_{n-1}), \ldots, (\epsilon_{n-k},\eta_{n-k}))$ and $(\epsilon_i,\eta_i) = (0,0)$ for $n-k \leq i \leq n-1$ and $\mathcal E_2 =((\epsilon_{n-k-1},\eta_{n-k-1}), \ldots, (\epsilon_0,\eta_0))$ with $\epsilon_{n-k-1} = 1$. We obtain that $N=\Psi(\mathcal D)=t_n+\eta_n l_n+\Psi(\mathcal E_2)$ and since $\Psi(\mathcal E_2)<l_n$, we know that $\eta_n$ is the same integer which we obtain by applying our algorithm to $N$. In particular, we have
\[\Psi^{-1}(\Psi(\mathcal D))=((1,\eta_n),\mathcal E_1,\Psi^{-1}(\Psi(\mathcal E_2)))=((1,\eta_n),\mathcal E)=\mathcal D.\]
\end{proof}

Using this digit expansions we are able to prove arithmetic properties of LS-sequences. We start with a lemma which provides conditions under which a point $\xi_{L,S}^N$ lies in a certain elementary interval.

\begin{lemma} \label{Lem:Carbone}
Let $N$ be an integer with representation given in \eqref{Rep:N}. Then
\begin{equation}\label{eq:Nth_point}
\xi_{L,S}^N=
\sum_{i=0}^n \left(\beta^i \min\{L,\epsilon_i+\eta_i\}+\beta^{i+1}\max\{\epsilon_i+\eta_i-L,0\}\right). \end{equation}

Moreover let $I_x^{(k)}=[\xi_{L,S}^x,\xi_{L,S}^x+\beta^k)$, with $x <l_k$, be an elementary interval. Then $\xi_{L,S}^N \in I_x^{(k)}$ if and only if 
\[x=\sum_{i=0}^{k-1} (\epsilon_it_i+\eta_il_i)\]
is the truncated representation of $N$.

In addition let $A_x^{(k)}(N)=\sharp \{m\: :\: m\leq N, \xi_{L,S}^m \in I_x^{(k)}\}$ and assume that
\[N=x+\sum_{i=k}^{n}(\epsilon_i t_{i}+\eta_i l_{i}).\]
Then 
\[A_x^{(k)}(N)= \sum_{i=0}^{n-k}(\epsilon_{i+k} t_{i}+\eta_{i+k} l_{i}) +1.\]
\end{lemma}

\begin{proof}
We start with the proof of \eqref{eq:Nth_point}. Due to Definition~\ref{Def:LSpoints} we have
\[\xi_{L,S}^N=\beta^n\min\{L,\epsilon_n+\eta_n\}+\beta^{n+1}\max\{\epsilon_n+\eta_n-L,0\}+\xi_{L,S}^{\tilde N}\]
with
\[\tilde N=\sum_{i=0}^{n-1} (\epsilon_it_i+\eta_il_i).\]
Repeating this argument inductively we will end up in \eqref{eq:Nth_point}.

Now let $N$ be of the form \eqref{Rep:N} and let
\[N_1 = \sum_{i=0}^{k-1} (\epsilon_it_i+\eta_il_i)\quad \text{and}
\quad N_2=\sum_{i=0}^{n-k} (\epsilon_{i+k}t_i+\eta_{i+k}l_i).\]
Then by \eqref{eq:Nth_point} we have $\xi_{L,S}^N=\xi_{L,S}^{N_1}+\beta^k\xi_{L,S}^{N_2}$. Since the LS-sequences only take values in the interval $[0,1)$ and since two distinct points of an LS-sequence with index $<l_k$ differ at least by $\beta^k$ this implies the second statement of the lemma. Moreover, by assumption $N_1 = x < l_k$ and therefore $N_2$ can take all integer values, since we have no restriction on the digits $\epsilon_k$ and $\eta_k$ (see Lemma~\ref{Lem:Digits}). Thus, $N_2$ counts all points $\xi_{L,S}^i$ in the interval $I_x^{(k)}$ with $x < i \leq N$. Since $\xi_{L,S}^x \in I_x^{(k)}$ is the first point that hits the interval $I_x^{(k)}$, the last statement of the lemma is established.
\end{proof}

Note that Lemma \ref{Lem:Carbone} would not hold in case of $l_k\leq x<t_k$, since this would imply $\epsilon_k=0$
(see Lemma~\ref{Lem:Digits}) and we would have serious restrictions for the digits of $N_2$. Thus $N_2$ could not take all integer values. 

Since a crucial point of the proof of Theorem~\ref{th1} is to have a precise knowledge of $A_x^{(k)}(N)$ the next lemma gives a further method to describe this quantity.

\begin{lemma}\label{Lem:ABcount}
Assume that $x<l_k$. If $\xi_{L,S}^N \in I_x^{(k)}$, then there exist integers $A,B$ such that $N=x+At_k+Bl_k$ and 
$A_x^{(k)}(N)=1+A+B$.
\end{lemma}

\begin{proof}
By Lemma~\ref{Lem:Carbone} we know that $N$ is of the form
\[N = x+\sum_{i=k}^n (\epsilon_i t_i+\eta_il_i).\]
Now by the recurrence relations for $t_k$ and $l_k$ and Lemma~\ref{Lem:MixedRec} we deduce that there are integers $A,B$ such that $N=x+At_k+Bl_k$.\\

Since $\xi_{L,S}^x\in I_x^{(k)}$ the proof of the lemma is complete, if we can prove that there exist integers $A,B$ such that
\[\sum_{i=k}^n (\epsilon_i t_i+\eta_il_i)=At_k+Bl_k\quad \text{and} \quad
\sum_{i=0}^{n-k} (\epsilon_{i+k}t_i+\eta_{i+k}l_i)=A+B,\]
where $\epsilon_i,\eta_i$ and $n\geq k$ are arbitrary non-negative integers.
We prove this assertion by induction on $n-k$. The case $n=k$ is trivial. We postpone checking the case $n=k+1$ to the end of the proof. Now, let us assume that $n-k\geq 2$. Using the recurrence relations \eqref{Rec:tk} and \eqref{Rec:lk} for $t_k$ and $l_k$ respectively we have 
\[\sum_{i=k}^n (\epsilon_i t_i+\eta_il_i)=\sum_{i=k}^{n-1} (\tilde\epsilon_i t_i+\tilde\eta_il_i)\]
and
\[\sum_{i=0}^{n-k} (\epsilon_{i+k} t_i+\eta_{i+k}l_i)=\sum_{i=0}^{n-k-1} (\tilde\epsilon_{i+k} t_i+\tilde\eta_{i+k}l_i),\]
with
\begin{gather*}
\tilde\epsilon_i=\epsilon_i, \tilde\eta_i=\eta_i \quad \text{for} \quad i=k,k+1,\ldots,n-3 \quad \text{and}\\
\tilde\epsilon_{n-2}=\epsilon_{n-2}+S\epsilon_n,\quad \tilde\epsilon_{n-1}=\epsilon_{n-1}+L\epsilon_n,\\
\tilde\eta_{n-2}=\eta_{n-2}+S\eta_n,\quad \tilde\eta_{n-1}=\eta_{n-1}+L\eta_n.
\end{gather*}
The new representations have fewer summands and by induction hypotheses we find appropriate integers $A$ and $B$.

It remains to check the case $n-k=1$. Using Lemma~\ref{Lem:MixedRec} we get
\begin{align*}
\epsilon_kt_k+&\eta_kl_k+\epsilon_{k+1}t_{k+1}+\eta_{k+1}l_{k+1}\\
=&\epsilon_kt_k+\eta_kl_k+\epsilon_{k+1}(t_k+l_k(L+S-1))+\eta_{k+1}(t_k+l_k(L-1))\\
=&t_k\stackrel{:=A}{\overbrace{(\epsilon_k+\epsilon_{k+1}+\eta_{k+1})}}+
l_k\stackrel{:=B}{\overbrace{(\eta_k+\epsilon_{k+1}(t_1-1)+\eta_{k+1}(l_1-1))}}.
\end{align*}
But, now
\begin{multline*}
A+B=\epsilon_k+\epsilon_{k+1}+\eta_{k+1}+\eta_k+\epsilon_{k+1}(t_1-1)+\eta_{k+1}(l_1-1))\\
=\epsilon_kt_0+\eta_kl_0+\epsilon_{k+1}t_1+\eta_{k+1}l_1
\end{multline*}
and we have found appropriate integers $A$ and $B$.
\end{proof}

The next lemma is related to the discrepancy of one-dimensional LS-sequences. In particular we are interested in an accurate formula for $\frac{A_x^{(k)}(N)}N$, where $\xi_{L,S}^N \in I_x^{(k)}$.

\begin{lemma}\label{Lem:Discrepancy-1-dim}
Assume that $N$ has a representation of the form \eqref{Rep:N}, and assume that $\xi_{L,S}^N \in I_x^{(k)}$. Then we have
\[\frac{A_x^{(k)}(N)}N =\beta^k+\frac{R(1-(-S\beta)^k)+1-x\beta^k}N,\]
where
\[R=\sum_{i=k}^{n}(\epsilon_i \tau_1+\eta_i \lambda_1)(-S\beta)^{i-k},\]
which can be estimated by
\[|R|<\max\{|\tau_1|,|\tau_1+(L+S-2)\lambda_1|\}\frac{1-(S\beta)^{n-k+1}}{1-S\beta}.\]
if $S\beta\neq 1$ and 
\[|R|<\max\{|\tau_1|,|\tau_1+(L+S-2)\lambda_1|\}\max\{n-k+1,0\}\]
if $S\beta=1$.
\end{lemma}

\begin{proof}
Using our assumptions and Lemma~\ref{Lem:Carbone} we can calculate the exact values of $A_x^{(k)}(N)$ and $N$. In fact, we have
\[N=x + \sum_{i=0}^{n}(\epsilon_i t_{i}+\eta_i l_{i})\quad \text{and}
\quad A_x^{(k)}(N)=\sum_{i=k}^{n}(\epsilon_i t_{i-k}+\eta_i l_{i-k}) +1.\]
This yields
\begin{align*}
\frac{A_x^{(k)}(N)}N =&\frac{\sum_{i=k}^{n}(\epsilon_i t_{i-k}+\eta_i l_{i-k}) +1}{\sum_{i=k}^{n}(\epsilon_i t_{i}+\eta_i l_{i}) +x}\\
=& \frac{\sum_{i=k}^{n}(\epsilon_i \tau_0+\eta_i \lambda_0)\beta^{-i+k}+\overset{R}{\overbrace{\sum_{i=k}^{n}(\epsilon_i \tau_1+\eta_i \lambda_1)(-S\beta)^{i-k}}}+1}
{\beta^{-k}\sum_{i=k}^{n}(\epsilon_i \tau_0+\eta_i \lambda_0)\beta^{-i+k}+(-S\beta)^k\underset{R}{\underbrace{\sum_{i=k}^{n}(\epsilon_i \tau_1+\eta_i \lambda_1)(-S\beta)^{i-k}}}+x}\\
=&\beta^k+\frac{R+1}N+\frac{\sum_{i=k}^{n}(\epsilon_i \tau_0+\eta_i \lambda_0)\beta^{-i+k}-N\beta^k}N\\
=&\beta^k+\frac{R+1}N+\frac{\sum_{i=k}^{n}(\epsilon_i \tau_0+\eta_i \lambda_0)\beta^{-i+k}}N\\
&-\frac{\left(\beta^{-k}\sum_{i=k}^{n}(\epsilon_i \tau_0+\eta_i \lambda_0)\beta^{-i+k}+(-S\beta)^kR+x\right)\beta^k}N\\
=&\beta^k+\frac{R(1-(-S\beta^2)^k)+1-x\beta^k}N.
\end{align*}
Thus it remains to estimate $R$. Note that $\tau_1<0$ and
\begin{equation*}
 \left| \epsilon_i \tau_i + \eta_i \lambda_i \right| < \max\{|\tau_1|,|\tau_1+(L+S-2)\lambda_1|\}
\end{equation*}
for $i = k, \ldots, n$. Hence to complete the proof of Lemma~\ref{Lem:Discrepancy-1-dim} we only have to compute the geometric sum 
\[\sum_{i=0}^{n-k}|-S\beta|^{i}.\]
We have to take absolute values in order to estimate $R$.
\end{proof}

Note that Lemma~\ref{Lem:Discrepancy-1-dim} gives a constant bound for $|R|$ for $n \rightarrow \infty$ if and only if $S \beta < 1$, which is equivalent to $L > S - 1$. This is exactly the case when we have a one-dimensional low-discrepancy LS-sequence.

\section{Proofs of main results}\label{Sec:ProofTh1}

We start with the proof of Theorem~\ref{th1}. According to Definition~\ref{def6}, we define the recurrences
\begin{align*}
 t_n^{(i)} &= L_i t_{n - 1}^{(i)} + S_i t_{n - 2}^{(i)},& t_0^{(i)}& =1, & t_1^{(i)}& =L_i+S_i,\\
 l_n^{(i)} &= L_i l_{n - 1}^{(i)} + S_i l_{n - 2}^{(i)},& l_0^{(i)}& =1, & l_1^{(i)}& =L_i,\\
 s_n^{(i)} &= L_i s_{n - 1}^{(i)} + S_i s_{n - 2}^{(i)},& s_0^{(i)}& =0, & s_1^{(i)}& =S_i,
\end{align*}
which correspond to the number of intervals, long intervals and short intervals after $n$ refinement steps in the $i$-th component of the two-dimensional LS-sequence, for $i=1,2$.\\

Now let $k$ and $m$ be integers satisfying the assumptions of Theorem~\ref{th1} and assume that the integers $\tilde{k}$ and $\tilde{m}$ are ``large'' (a precise condition will be given later). Furthermore choose two integers $x_1<l^{(1)}_k$ and $x_2<l^{(2)}_m$ such that $x_1\neq x_2$. In the sequel we will consider the intervals $I=I_{x_1}^{(k)} \times I_{x_2}^{(m)}$ and $\tilde{I} = I_{x_1}^{(k+\tilde{k})} \times I_{x_2}^{(m+\tilde{m})}$. We want to prove that no point of the two-dimensional LS-sequence is contained in the interval $\tilde{I}$. Note that $\tilde{I} \subset I$, and consequently a point can only be contained in $\tilde{I}$ if it is also contained in $I$. Now let $N$ be given, and assume that $\vect{\xi}_\mathcal{B}^N \in I$. We will also assume that $\vect{\xi}_\mathcal{B}^N \in \tilde{I}$, and show that this leads to a contraction (provided $\tilde k$ and $\tilde m$ are sufficiently large) Consequently, no point of the sequence $(\vect{\xi}_\mathcal{B}^n)_{n \in \mathbb{N}}$ can be contained in $\tilde{I}$.\\

Due to Lemma~\ref{Lem:ABcount}, $\vect{\xi}_{L_1,S_1}^N \in I_{x_1}^{(k)}$ implies that there exist integers $A_1,B_1$ such that $N=x_1+A_1t^{(1)}_k+B_1l^{(1)}_k$. Similarly, $\vect{\xi}_{L_2,S_2}^N \in I_{x_2}^{(m)}$ implies the existence of integers $A_2,B_2$ such that $N=x_2+A_2t^{(2)}_m+B_2l^{(2)}_m$. Thus
\begin{equation}\label{eq:N}
x_1+A_1t^{(1)}_k+B_1l^{(1)}_k=x_2+A_2t^{(2)}_m+B_2l^{(2)}_m.
\end{equation}
Moreover, choosing $A_1$ and $B_1$ according to Lemma~\ref{Lem:ABcount} we know that there are exactly $1+A_1+B_1$ points with index $\leq N$ lying in the interval $I_{x_1}^{(k)}$. Therefore Lemma~\ref{Lem:Discrepancy-1-dim} yields
\[\frac{A_1+B_1+1}{x_1+A_1t^{(1)}_k+B_1l^{(1)}_k}=\beta_1^k+\frac{R_1(1-(-S_1\beta_1^2)^k)+1-x_1\beta_1^k}{N}.\]
Multiplying both sides with $N=x_1+A_1t^{(1)}_k+B_1l^{(1)}_k$ and solving for $B_1$ we obtain
\begin{equation}\label{eq:B1}
B_1=A_1\frac{t^{(1)}_k\beta_1^k-1}{1-l^{(1)}_k\beta_1^k}+\frac{R_1(1-(-S_1\beta_1^2)^k)}{1-l^{(1)}_k\beta_1^k}.
\end{equation}
A similar argument for the second component yields
\begin{equation}\label{eq:B2}
B_2=A_2\frac{t^{(2)}_m\beta_2^m-1}{1-l^{(2)}_m\beta_2^m}+\frac{R_2(1-(-S_2\beta_2^2)^m)}{1-l^{(2)}_k\beta_2^m}.
\end{equation}
Now we resubstitute equations \eqref{eq:B1} and \eqref{eq:B2} into \eqref{eq:N} and obtain
\begin{multline}\label{eq:A1A2}
x_1+A_1\frac{t^{(1)}_k-l^{(1)}_k}{1-l^{(1)}_k\beta_1^k}+\frac{R_1l^{(1)}_k(1-(-S_1\beta_1^2)^k)}{1-l^{(1)}_k\beta_1^k}=
\\
x_2+A_2\frac{t^{(2)}_m-l^{(2)}_m}{1-l^{(2)}_m\beta_2^m}+\frac{R_2l^{(2)}_m(1-(-S_2\beta_2^2)^m)}{1-l^{(2)}_m\beta_2^m}. 
\end{multline}

Now let us investigate the quantities $\frac{t_k-l_k}{1-l_k\beta^k}$ and
$\frac{l_k(1-(-S\beta^2)^k)}{1-l_k\beta^k}$.

\begin{lemma}\label{Lem:Aux}
We have
\[\frac{t_k-l_k}{1-l_k\beta^k}=\beta^{-k-1}\]
and
\[\frac{l_k(1-(-S\beta^2)^k)}{1-l_k\beta^k}=\frac{\lambda_0}{\lambda_1}\beta^{-k}+(-S\beta)^k.\]
\end{lemma}

\begin{proof}
Using the explicit formulas \eqref{Rec:tk} and \eqref{Rec:lk} for the recurrences $t_k$ and $l_k$, respectively, we obtain
\begin{align*}
\frac{t_k-l_k}{1-l_k\beta^k}=&
\frac{(\tau_0-\lambda_0)\beta^{-k}+(\tau_1-\lambda_1)(-S\beta)^k}{1-\beta^k(\lambda_0\beta^{-k}+\lambda_1(-S\beta)^k)}
=\frac{\frac{S}{\sqrt{L^2+4S}}(\beta^{-k}-(-S\beta)^k)}{\lambda_1(1-(-S\beta^2)^k)}\\
=&\frac{\frac{S}{\sqrt{L^2+4S}}}{\lambda_1}\cdot\beta^{-k}\cdot\frac{1-(-S\beta^2)^k}{1-(-S\beta^2)^k}=\beta^{-k-1},
\end{align*}
which proves the first part of Lemma~\ref{Lem:Aux}.\\

Note that 
\[\frac{\lambda_1}{\frac{S}{\sqrt{L^2+4S}}}=\frac{\frac{-L+\sqrt{L^2+4S}}{2\sqrt{L^2+4S}}}{\frac{S}{\sqrt{L^2+4S}}}= \frac{-L+\sqrt{L^2+4S}}{2S}=\beta.\]

Let us now prove the second statement of the lemma. Again we use the explicit formulas \eqref{Rec:tk} and \eqref{Rec:lk} and obtain
\begin{align*}
\frac{l_k(1-(-S\beta^2)^k)}{1-l_k\beta^k}=&
\frac{l_k(1-(-S\beta^2)^k)}{1-\beta^k(\lambda_0\beta^{-k}+\lambda_1(-S\beta)^k)}\\
=&\frac{l_k(1-(-S\beta^2)^k)}{\lambda_1(1-(-S\beta^2)^k)}\\
=&\frac{\lambda_0}{\lambda_1}\beta^{-k}+(-S\beta)^k
\end{align*}
\end{proof}

Continuing the proof of Theorem~\ref{th1}, let us insert the explicit formulas of Lemma~\ref{Lem:Aux} into \eqref{eq:A1A2}. We get
\begin{equation}\label{eq:A1}
A_1=A_2\frac{\beta_1^{k+1}}{\beta_2^{m+1}}+(x_2-x_1)\beta_1^{k+1}+\tilde R,
\end{equation}
where
\begin{equation}\label{eq:R-Fehler}
\tilde R=R_2\stackrel{c_2:=}{\overbrace{\beta_2\left(\frac{\lambda_0^{(2)}}{\lambda_1^{(2)}}+(-S_2\beta_2^2)^m \right)\frac{\beta_1^{k+1}}{\beta_2^{m+1}}}}-
R_1\stackrel{c_1:=}{\overbrace{\beta_1\left(\frac{\lambda_0^{(1)}}{\lambda_1^{(1)}}+(-S_1\beta_1^2)^k)\right)}}.
\end{equation}

Now by assumption we have $\frac{\beta_1^{k+1}}{\beta_2^{m+1}}=\frac pq$ for some coprime positive integers $p,q$ and
\eqref{eq:A1} can be written as
\begin{equation}\label{eq:Integers}
2(qA_1-pA_2)=2(x_2-x_1)q\beta^{k+1}+2q\tilde R,
\end{equation}
Obviously the left side of \eqref{eq:Integers} is an even integer. In order to prove Theorem~\ref{th1} we want to show that the right side is not an even integer if $\vect{\xi}_\mathcal{B}^N \in\tilde I$ and $\tilde k$ and $\tilde m$ are sufficiently large. By our assumptions on $x_1,x_2$ and the parameters $L_1,S_1,L_2$ and $S_2$ we have
\begin{equation}\label{x1x2}
\|2q(x_2-x_1)\beta_1^{k+1}\|_{\text{odd}}=\epsilon<1,
\end{equation}
where $\| \cdot \|_{\text{odd}}$ denotes the distance to the nearest odd integer. Note that $\beta_1^{k+1}$ is irrational due to Lemma~\ref{lem:irrational} and therefore indeed $\epsilon<1$.

Now Lemma~\ref{Lem:Discrepancy-1-dim} tells us that the assumption $\xi_{\mathcal B}^N\in \tilde I$ yields
\[N=x+\sum_{i=k+\tilde k}^{n}(\epsilon^{(1)}_i t^{(1)}_i+\eta^{(1)}_i l^{(1)}_i)(-S_1\beta_1)^{i},\]
and we have
\[|R_1|\leq \max\{|\tau^{(1)}_1|,|\tau^{(1)}_1+(L_1+S_1-2)\lambda^{(1)}_1|\}\frac{\left|- S_1 \beta_1 \right|^{\tilde k}}{1-S_1\beta_1}.\]
A similar inequality also holds for $|R_2|$. Now we choose $\tilde k$ sufficiently large such that 
\[\left| -S_1 \beta_1\right|^{\tilde k}<\frac{1-\epsilon}{|c_1|4q \max\{|\tau^{(1)}_1|,|\tau^{(1)}_1+(L_1+S_1-2)\lambda^{(1)}_1|\}}(1-S_1\beta_1)\]
in case of $|c_1|\neq 0$ and $\tilde k=0$ otherwise. Similarly we choose $\tilde m$ sufficiently large such that
\[\left| -S_2 \beta_2\right|^{\tilde m}<\frac{1-\epsilon}{|c_2|4q \max\{|\tau^{(2)}_1|,|\tau^{(2)}_1+(L_2+S_2-2)\lambda^{(2)}_1|\}}(1-S_2\beta_2)\]
in case of $|c_2|\neq 0$ and $\tilde m=0$ otherwise.
With this choice of $\tilde k$ and $\tilde m$ we obtain
$$
|\tilde R|<\frac {1-\epsilon}{2q},
$$
which together with \eqref{x1x2} implies that the right side of \eqref{eq:Integers} is an odd integer plus something $<1$ in absolute values. Consequently \eqref{eq:A1} does not have any solution $A_1, A_2$. However, since \eqref{eq:A1} followed from our assumptions, this means that we have a contradiction. Consequently it is not possible that $\vect{\xi}_\mathcal{B}^N \in \tilde{I}$ for any $N$, which means that $\tilde{I}$ does not contain any point of $(\vect{\xi}_\mathcal{B}^n)_{n \in \mathbb{N}}$. This completes the proof of Theorem~\ref{th1}.\\

We continue with the proof of Theorem~\ref{th2}. Let $b = \gcd(L_1, S_1, L_2, S_2) > 1$ and let $t_n^{(1)}, l_n^{(1)}, t_n^{(2)}, l_n^{(2)}$ be defined as in the proof of Theorem~\ref{th1}. Note that $b$ divides $t_n^{(1)}, l_n^{(1)}, t_n^{(2)}, l_n^{(2)}$ for all $n \geq 1$. Now consider the interval $I = [0, \beta_1) \times [\beta_2, 2 \beta_2)$. Note that $2 \beta_2 < 1$ since by $\gcd(L_1, S_1, L_2, S_2) > 1$ it follows that $L_2 \geq 2$. It follows from Lemma~\ref{Lem:ABcount} that we can write every $N_1$ for which $\xi_{L_1, S_1}^{N_1} \in [0, \beta_1)$ as $N_1 = 0 + A_1 t_1^{(1)} + B_1 l_1^{(1)}$ and every $N_2$ for which $\xi_{L_2, S_2}^{N_2} \in [\beta_2, 2 \beta_2)$ as $N_2 = 1 + A_2 t_1^{(2)} + B_2 l_1^{(2)}$ for appropriate integers $A_1, B_1, A_2, B_2$. Hence we obtain for every $N_1, N_2$ that $N_1 \equiv 0 \bmod{b}$ and $N_2 \equiv 1 \bmod{b}$, respectively. Consequently no point of the two-dimensional LS-sequence can be contained in $I$. This proves Theorem~\ref{th2}.



\section*{Acknowledgements}
The authors thank Maria Rita Iac\`o for presenting numerical studies of non-uniformly distributed LS-sequences in her talk at TU Graz in June 2012.

\end{document}